\documentclass[]{amsart}
\usepackage[bottom]{footmisc}
\usepackage{xr}
\usepackage[dvipsnames,table,xcdraw]{xcolor}

\calclayout

\usepackage[colorlinks=true]{hyperref}
	\hypersetup{urlcolor=RubineRed,linkcolor=RoyalBlue,citecolor=ForestGreen}

\usepackage{cooltooltips}
\usepackage{graphicx}

\usepackage{forest,adjustbox}
\usepackage{forest}
\usetikzlibrary{shapes.geometric,decorations.markings}
\usetikzlibrary{positioning,arrows.meta,decorations.pathreplacing}

\forestset{%
  my tree/.style={
    for tree={
      circle,
      draw,
      inner sep=1pt,
      l sep'=5mm,
      l'=5mm,
      s sep'=2mm,
    },
  }
}
\usepackage{caption} 
			
	\newtheorem{theorem}{Theorem}[subsection]
	\newtheorem{lemma}[theorem]{Lemma}
	\newtheorem{proposition}[theorem]{Proposition}
	\newtheorem{corollary}[theorem]{Corollary}
	\newtheorem*{theorem*}{Theorem}
	\newtheorem*{lemma*}{Lemma}
	\newtheorem*{proposition*}{Proposition}
	\newtheorem*{corollary*}{Corollary}
		\theoremstyle{definition}
		\newtheorem{conjecture}[theorem]{Conjecture}
		\newtheorem{definition}[theorem]{Definition}
		
		\newtheorem{remark}[theorem]{Remark}

\numberwithin{equation}{section}

\usepackage[bottom]{footmisc}
\usepackage{tikz}
\usepackage{tikz-cd}
\usepackage{mathrsfs}
\usepackage{supertabular}
\usepackage[shortlabels]{enumitem}
\usetikzlibrary{lindenmayersystems}
\DeclareMathOperator{\Gal}{Gal}
\DeclareMathOperator{\GL}{GL}

\DeclareMathOperator{\Spec}{Spec} 
\DeclareMathOperator{\Tr}{Tr} 
\DeclareMathOperator{\ext}{ext} 
\DeclareMathOperator{\Hom}{Hom} 
\DeclareMathOperator{\End}{End} 

\newcommand{\B}{\mathbf{B}}

\newcommand{\p}{\mathfrak{p}}

\newcommand{\cA}{\mathcal{A}}

\newcommand{\cB}{\mathcal{B}}

\newcommand{\hra}{\hookrightarrow}

\renewcommand{\P}{\mathbf{P}}

\renewcommand{\O}{\mathcal{O}}

\DeclareFontEncoding{OT2}{}{} 

\renewcommand{\H}{\mathbf{H}}

\DeclareMathOperator{\red}{red}

\DeclareMathOperator{\Fr}{Fr}

\DeclareMathOperator{\Norm}{Norm}

\DeclareMathOperator{\Tbf}{\mathbf{T}}

\DeclareMathOperator{\Pbf}{\mathbf{P}}

\DeclareMathOperator{\Mbf}{\mathbf{M}}
\DeclareMathOperator{\Sbf}{\mathbf{S}}

\DeclareMathOperator{\Res}{\text{Res}}

\DeclareMathOperator{\Sh}{Sh}
\newcommand*\iso{
  \xrightarrow{\raisebox{-0.25 em}{\smash{\ensuremath{\sim}}}}%
}
\usepackage{dsfont}
\usepackage{upgreek}
\usepackage{stmaryrd}
\usepackage{color}
\usepackage{amsthm,amssymb,amsmath}
\usepackage{lastpage}
\usepackage{centernot}
\usepackage{graphicx}
\usetikzlibrary{cd}
\usetikzlibrary{trees}
\usetikzlibrary{fit}

\tikzset{%
  highlight/.style={rectangle,rounded corners,fill=blue!15,draw,fill opacity=0.5,thick,inner sep=0pt}
}

\usepackage{nicematrix}
\usepackage{url}
\usetikzlibrary{hobby,backgrounds,calc,trees}
\usetikzlibrary{matrix,arrows,decorations.pathmorphing}

\def\1{\mathds{1}}

\def\cC{\mathcal{C}}

\def\cO{\mathcal{O}}
\def\cH{\mathcal{H}}

\def\fc{\mathfrak{c}}

\def\A{\mathds{A}}

\def\Z{\mathds{Z}}

\def\Q{\mathds{Q}}
\def\R{\mathds{R}}
\def\C{\mathds{C}}

\def\T{\mathbf{T}}

\def\U{\mathbf{U}}
\def\G{\mathbf{G}}
\def\H{\mathbf{H}}

\def\Hom{\text{Hom}}

\def\Spec{\text{Spec}}
\def\Gal{\text{Gal}}

\newcommand{\authnote}[2][]{\noindent {\if!#1!  {\bf TODO} \else {\small \bf #1} \fi: #2}}

\newcounter{tasknumber}[subsection]
\newcommand{\task}[2][]{%
  \addtocounter{tasknumber}{1}%
  \begin{center}%
  \framebox[1.1\width]{\begin{minipage}{0.9\textwidth}%
  \textbf{Task \arabic{tasknumber}} \textit{\if!#1(unassigned)!\else (#1)\fi}: {#2}%
  \end{minipage}}%
  \end{center}%
}

\newcounter{assumptionnumber}
\newcommand{\assumption}[2][]{%
  \addtocounter{assumptionnumber}{1}%
  \begin{center}%
  \framebox[1.1\width]{\begin{minipage}{0.9\textwidth}%
  \textbf{Assumption \arabic{assumptionnumber}} \textit{\if!#1!\else (#1)\fi}: {#2}%
  \end{minipage}}%
  \end{center}%
}

\usepackage[scr=boondoxo,scrscaled=1.05]{mathalfa}
\usepackage[intoc,refpage]{nomencl}
\def\@@@nomenclature[#1]#2#3{%
\def\@tempa{#2}\def\@tempb{#3}%
\protected@write\@glossaryfile{}%
{\string\glossaryentry{#1\nom@verb\@tempa @[{\nom@verb\@tempa}]%

nompageref{\begingroup\nom@verb\@tempb\protect\nomeqref{\theequation}}}%
{\thepage}}%
 \endgroup
 \@esphack}

\usepackage[intoc,refpage]{nomencl}
\def\@@@nomenclature[#1]#2#3{%
\def\@tempa{#2}\def\@tempb{#3}%
\protected@write\@glossaryfile{}%
{\string\glossaryentry{#1\nom@verb\@tempa @[{\nom@verb\@tempa}]%

nompageref{\begingroup\nom@verb\@tempb\protect\nomeqref{\theequation}}}%
{\thepage}}%
 \endgroup
 \@esphack}

\usepackage{ifthen}
\renewcommand{\nomgroup}[1]{%
\ifthenelse{\equal{#1}{A}}{\item[\textbf{Chapter II \S 1 \& \S 2}]}{%
\ifthenelse{\equal{#1}{B}}{\item[\textbf{Chapter II \S 3}]}{%
\ifthenelse{\equal{#1}{C}}{\item[\textbf{Chapter III}]}{
\ifthenelse{\equal{#1}{D}}{\item[\textbf{Chapter IV}]}{
\ifthenelse{\equal{#1}{E}}{\item[\textbf{Chapter V}]}{
\ifthenelse{\equal{#1}{F}}{\item[\textbf{Chapter VI}]}{
\ifthenelse{\equal{#1}{G}}{\item[\textbf{Chapter VII}]}{
\ifthenelse{\equal{#1}{H}}{\item[\textbf{Chapter V}]}{
}}}}}}}}}
\makenomenclature

\usepackage{apptools}
\AtAppendix{\counterwithin{theorem}{section}}

\usepackage{comment}
\usepackage{aligned-overset}
  
\begin{document}

\title{\textsc{Seed Relations \\ for \\Eichler--Shimura congruences and Euler systems}}


\author{Reda Boumasmoud}
\address{Institut de Mathematiques de Jussieu-Paris Rive Gauche. 75005 Paris, France}
\curraddr{}
\email{reda.boumasmoud@imj-prg.fr \& reda.boumasmoud@gmail.com}
\thanks{The author was supported by the Swiss National Science Foundation grant \#PP00P2-144658 and \#P2ELP2-191672.}

\makeatletter 
\@namedef{subjclassname@2020}{%
  \textup{2020} Mathematics Subject Classification}
\makeatother



\date{}

\begin{abstract}
This paper proves that the $\mathbb{U}$-operator \cite{UoperatorsII2021} attached to a cocharacter is a right root of the corresponding Hecke polynomial. 
This result is an important ingredient in the proof of (i) the horizontal norm relations in the context of Gross--Gan--Prasad cycles and of (ii) the generalization of Eichler--Shimura relations.
\end{abstract}


 \maketitle


\tableofcontents

\section{Introduction}
\subsection{Motivation}
In 1954, Eichler discovered the first instance of the link between zeta functions of Shimura varieties and automorphic
L-functions. Shortly thereafter, Shimura extended Eichler results to compute the zeta functions of quaternionic curves. 
Their work was based on the congruence relation, known now as the Eichler--Shimura relation, which played
an important role in the theory of arithmetic of elliptic curves and modular forms. Later on, in the 70s, Langlands launched a program that aims to generalize the previous work to compute zeta functions attached to all Shimura varieties. Gradually a conjecture generalizing the Eichler--Shimura relation has emerged and was formulated by Blasisus and Rogawski \cite[\S 6]{BR94}. We give its statement below after setting some background.

Let $\G$ be a connected, reductive group defined over $\Q$ and let $\mathds{S}=\Res_{\C/\R}\mathds{G}_{m,\C}$. 
Suppose we have a homomorphism of algebraic $\R$-groups $\mathds{S}\to \G_\R$, which satisfies the axioms of Deligne \cite[Definition 5.5]{milne:shimura}. 
Let $K$ be an open compact subgroup of $\G(\A_f)$ of the form $\prod_{v < \infty}K_v$, where $K_v\subset \G(\Q_v)$ and $K_v$ is hyperspecial for almost all the finite places $v$. 
This gives rise to the Shimura variety $\Sh_K(\G,\mathcal{X})$ with reflex field $E$ and whose complex points are $$\Sh_K(\G,\mathcal{X})(\C)=\G(\Q)\backslash\mathcal{X}\times \G(\A_f)/K.$$ 
Assume that $K$ is sufficiently small, so that $\Sh_K(\G,\mathcal{X})$ is a smooth. We fix a prime $p$ over which $\G$ is unramified and the level structure $K$ has the form $K^p K_p$ with $K_p$ hyperspecial. For each prime ideal $\p$ of $E$ lying over $p$, Blasius and Rogawski have defined a polynomial $H_\p\in \mathcal{H}(\G(\Q_p)// K_p,\Q)[X]$, and they conjectured that:
\begin{conjecture}[Blasius--Rogawski]\label{conj1}Let $\ell$ be a prime $\neq p$ (i) The Shimura variety $\Sh_K(\G,\mathcal{X})$ has good reduction at $\p$ (in some sense); and (ii) we have $$H_\p(\Fr_\p) = 0 \text{ in the ring }\End_{\Q_\ell}(H_{\text{\'{e}t}}^\bullet(\Sh_K(\G,\mathcal{X})\times_E\overline{\Q},\Q_\ell)).$$
\end{conjecture} 
This conjecture was proved by Ihara (extending cases treated by Eichler and Shimura) for Shimura curves. The first statement has been established for Shimura varieties of abelian type by Kisin \cite{kisin:jtnb,kisin:jams} and the second part was proved by: B\"{u}ltel for certain orthogonal groups \cite{Bu97}, Wedhorn \cite{wedhorn00} in the PEL case for groups that are split over $\Q_p$, B\"ultel--Wedhorn for the
unitary case of signature $(n -1, 1)$ with $n$ even \cite{bultel-wedhorn}, Koskivirta for a unitary similitude group of signature $(n - 1, 1)$ over $\Q$ when $p$ is inert in the reflex field and $n$ odd \cite{koskivirta:congruence} and finally H. Li showed recently the conjecture for simple GSpin Shimura varieties \cite{HLi2018}.

In all these cases for which the conjecture is known, the authors prove a slightly stronger version of it where the desired annihilation is taking place in a "geometric" ring of correspondences or cycles in characteristic $p$. 
Assume that $\Sh_K(\G,\mathcal{X})$ is of Hodge-type and let $\mathscr{S}_K$ be its integral model over $\cO_{E_{\p}}$. This scheme has an interpretation as a moduli space of abelian schemes with additional structures. Following Chai--Faltings \cite{faltings-chai}, Moonen defines in \cite{moonen:serre-tate} a stack $p-\text{Isog}$ over $\cO_{E_{\p}}$, parametrizing $p$-isogenies between two points of $\mathscr{S}_K$. It has two natural projections to $\mathscr{S}_K$, sending an isogeny to its target and source. 
the subalgebra generated by the irreducible components.
Consider the $\Q$-algebra of cycles $\Q[p-\text{Isog} \times E]$  and $\Q[p-\text{Isog} \times k_{\cO_{E_\p}}]$ where $k_{\cO_{E_\p}}$ is the residue field of $\cO_{E_\p}$, 
here multiplication is defined by composition of isogenies. Define $p-\text{Isog}^{\text{ord}} \times k_{\cO_{E_\p}}$ as the preimage of the $\mu$-ordinary locus of
the special fiber of the $\mathscr{S}_K$, under the source projection. 
We get a diagram of $\Q$-algebra homomorphism
$$\begin{tikzcd}
\mathcal{H}(\G(\Q_p)// K_p,\Q) \arrow{r}{h}\arrow{dd}{\dot{\mathcal{S}}_M^G} &\Q[p-\text{Isog} \times E ]\arrow{d}{\sigma}\\
&\Q[p-\text{Isog} \times k_{\cO_{E_\p}} ]\arrow[d,shift left,"\text{ord}"]\\
\mathcal{H}(\Mbf(\Q_p)// K_p\cap M(\Q_p),\Q)\arrow{r}{\bar{h}}&\Q[p-\text{Isog}^{\text{ord}} \times k_{\cO_{E_\p}}]\arrow{u}{\text{cl}}
\end{tikzcd}$$
where the big square is commutative, $\Mbf$ is the centralizer of the norm of the dominant coweight $\mu$ given by the Shimura datum, the homomorphism $\dot{\mathcal{S}}_M^G$ is the untwisted Satake transform, $\sigma$ is the specialization map of cycles, the map $\text{ord}$ intersects a cycle with the ordinary $\mu$-locus while \text{cl} is the map sending a cycle to its closure. 
There is a natural Frobenius section of the source projection, mapping an abelian variety to its Frobenius isogeny, which produces a closed subscheme $F$ of $p-\text{Isog} \times k_{\cO_{E_\p}} $. 
\begin{conjecture}\label{conj2}
The cycle $F$ is a root of the polynomial $$\sigma \circ h (H_\p)(X) \in \Q[p-\text{Isog} \times k_{\cO_{E_\p}} ][X].$$\end{conjecture}
Functorial properties of cohomology shows that Conjecture \ref{conj2} implies Conjecture \ref{conj1}. 
Most known cases of Conjecture \ref{conj2} are obtained by proving first the conjecture on the generically ordinary $p$-isogenies. This reduces to B\"ultel's group theoretic result which says that we have an annihilation 
\begin{equation}\label{bultrel}H_\p(\mu)=0 \text{ in the $\Q$-algebra }\mathcal{H}(\Mbf(\Q_p)// K_p\cap \Mbf(\Q_p),\Q).\tag{$\star$}\end{equation} 
Now, If the ordinary locus $p-\text{Isog}^{\text{ord}} \times k_{\cO_{E_\p}}$ is dense in $p-\text{Isog} \times k_{\cO_{E_\p}}$, then B\"ultel's argument is sufficient to prove the full congruence conjecture. This is the cases studied by Chai-Faltings, B\"ultel, Wedhorn and B\"ultel--Wedhorn.

We have a commutative diagram:
$$\begin{tikzcd}
\cH_K(\Q)\arrow[equal]{d} \arrow[hookrightarrow]{rr}{\dot{\mathcal{S}}_M^G}&& \cH(\Mbf(\Q_p)\sslash  K_p\cap \Mbf(\Q_p),\Q)\arrow[d,equal]\\
\End_{\Z[G]}\Q[G/K]\arrow[hook]{r}&\End_{\Z[P]}\Q[G/K]\arrow{r}&\End_{\Z[\Mbf(\Q_p)]}\Q[\Mbf(\Q_p) /K_p\cap \Mbf(\Q_p)].
\end{tikzcd}$$
Our main results (Theorem \ref{uoproothecke}) shows in particular that Bültel's relation (\ref{bultrel}) lifts naturally to an analogous relation
\begin{equation*}\label{myequality}
H_\p(u_\mu)=0 \in \End_{\Q[\Pbf(\Q_p)]}\Q[\G(\Q_p)/K_p]\tag{$\dag$}
\end{equation*}
where $u_\mu$ is the $\mathds{U}$-operator attached to $\varpi^{\mu}$ \cite{UoperatorsII2021} and $\Pbf$ is the largest parabolic subgroup of $\G$ relative to which $\mu$ is dominant. 
For applications, a key advantage of the latter relations (upon B\"{u}ltel's) is that while $\mathcal{H}(\Mbf(\Q_p)// K_p\cap \Mbf(\Q_p),\C)$ still had to be made acting on various spaces, the non-commutative ring $\End_{\Z[\Pbf(\Q_p)]}(\G(\Q_p)/K_p)$ already acts (faithfully and by definition) on the ubiquitous space $\Q[\G(\Q_p)/K_p]$.

{In a work in progress the author is tackling (using \ref{myequality} instead) a generalization of Conjecture \ref{conj2} for abelian-type Shimura varieties \cite{ESboumasmoud}.}

\subsection{Main result}

Let $F$ be  a finite extension of $\Q_p$ for some prime $p$,  $\O_F$ its ring of integers, $\varpi$ a fixed uniformizer in $\O_F$ and $k_F$ the residue field of $F$ of size $q$. For every scheme $X$ over $\Spec \, \O_F$, we set $X_{\kappa(F)}:=X \times_{\Spec \,\O_F}\Spec \,  \kappa(F)$ for the special fiber.

Let $\G/F$ be an unramified reductive group, $\mathbf{S}$ a maximal $F$-split subtorus of $\G$ and $\cA$ the apartment attached to $\Sbf$ in the extended Bruhat--Tits building of $\G$, together with a fixed origin a hyperspecial point $a_\circ\in \cA$. 
Let $\T$ be the centralizer of $\Sbf$, which is a maximal $F$-torus in $\G$, $\B= \T \cdot \U^+$ a Borel subgroup with unipotent radical $\U^+$ and $W=N_\G(\Sbf)(F)/\T(F)$ be the Weyl group. 

Let $K$ be a hyperspecial maximal open compact subgroup of $\G$ attached $a_\circ$.  
Bruhat and Tits attach to $a_\circ$ a reductive $\cO_F$-model $\mathcal{G}$ of $\G$. Let $K$ be the corresponding parahoric subgroup, i.e. $\mathcal{G}(\cO_F)$. This also applies to the reductive group $\Tbf$ and $a_\circ$, we get then a reductive $\cO_F$-model $\mathcal{T}$ of $\T$. Let $I$ be the Iwahori subgroup that is defined by
$$I=\{g\in \G(\cO_F):\red(g)\in \B(k_F)\}.$$ 

For any algebraic $F$-groups $\H$ (bold style), we denote its group of $F$-points by the ordinary capital letter $H=\mathbf{H}(F)$.

Let $\nu_N^\prime\colon N \to (X_*(\Sbf)\otimes_\Z \R) \rtimes W$ be the map characterized by  
$$\nu_N^\prime({\varpi^{\lambda}})=\lambda.$$
Note that $\nu_N^\prime=-\nu_N$, where $\nu_N$ is the Bruhat--Tits translation homomorphism. 
Set $T_1:= \mathcal{T}(\cO_F)=\ker \nu_N = \ker  \kappa_\T$, 
where $\kappa_{\T}$ is the Kottwitz homomorphism\footnote{Note that in this unramified case, $\T$ splits over the completion of $F^{un}$ denoted previously by $L$. Thus, the Kottwitz homomorphism takes the simpler form $\kappa_{\T}\colon \T(L) \to X_*(\T)$.}. 
We embed $X_*(\Sbf)$ into $T$ (using $\nu_N^\prime$) by identifying $\lambda \in X_*(\Sbf)$ with $\varpi^{\lambda}:=\lambda(\varpi)$. 
Using this identification, we have 
$$\Lambda_T:= T/T_1 \simeq X_*(\T)_F\simeq X_*(\Sbf).$$
Set $\Phi^+$ for the set of $B$-positive roots, the one that appears in $\text{Lie}(B)$, or equivalently if it takes positive values on the vectorial chamber $\cC^-$ opposite to $\cC^+$; where $\cC^+$ is the vectorial chamber corresponding to $B$\footnote{Given $B$, the chamber $\cC^+$ is the unique vectorial chamber with apex $a_\circ$ for which $T_1 U^+$ is the union of the fixators of all quartiers $a + \cC^+$ with $a\in \cA$.}. 

We say that $\lambda\in X_*(\Sbf)$ is $\B$-dominant if $\langle \lambda, \alpha \rangle\ge 0$ for all $\alpha \in \Phi^+$. Let ${\overline{\cC}}\subset \cA_{\ext}$ denotes the closed vectorial chamber corresponding to the Borel $B$ in the extended apartment attached to $\Sbf$. Thus, an element $t={\varpi^{\lambda}}$ for $\lambda\in X_*(\Sbf)$ is antidominant if and only if $\lambda \in X_*(\Sbf) \cap \overline{\cC}$, if and only if $\lambda$ is $\B$-dominant, since 
$\langle \nu_N^\prime(t),\alpha \rangle= \langle \lambda ,\alpha \rangle\le 0, \forall \alpha \in \Phi^+$. 
Write $\Lambda_T^-$ for the set of antidominant elements in $\Lambda_T$.

For any extension $E$ of $F$, let $\mathcal{M}(E)\nomenclature{$\mathcal{M}(E)$}{The $\G(E)$-conjugacy
classes of cocharacters $\mathds{G}_{m,E}\to \G_{E}$}$ be the set of $\G(E)$-conjugacy classes of (algebraic group) cocharacters $\mathds{G}_{m,E}\to  \G_{E}$. By \cite[Lemma 1.1.3]{Ko1}, the canonical surjective morphism $X_*(\Sbf) \to \mathcal{M}(F)$ yields the following identification
$$X_*(\Sbf)/W(\G,\Sbf) \simeq \mathcal{M}(F) \simeq \mathcal{M}(\overline{F})^{\Gal(\overline{F}/F)} \simeq \left(X_*(\T)/W(\G_{\overline{F}},\T)\right)^{\Gal(\overline{F}/F)}.$$
In addition, using the Cartan decomposition one gets another identification identification 
$$\mathcal{M}(F)\simeq K\backslash G/K,$$
given by $[\lambda]\mapsto K{\varpi^{\lambda}}K$.

Let ${\fc} \in \mathcal{M}(\overline{F})\nomenclature[D]{$\fc$}{Fixed conjugacy class in $\mathcal{M}(\overline{F})$}$ and $F({\fc})\subset F^{un}$ its field of definition. 
Set $d=[{F(\fc)}:F]$. 
Let $\mu\in \Norm_{F(\fc)/F}{\fc}$ be the cocharacter of $\T$ which is $\B$-dominant, i.e. $\varpi^{\mu}$ is antidominant. 
Let $\P$ be the largest parabolic subgroup of $\G$ relative to which $\mu$ is dominant, $\mathbf{L}$ is a Levi factor of $\P$ (which is also the centralizer of $\mu$ in $\G$) and $\U_P^+$ the unipotent radical of $\P$. 

In \cite{UoperatorsII2021}, to any element $t\in \Lambda_T^-$ is attached an operator $u_t\in \End_{\Z[B]}\Z[G/K]$ characterized by sending the trivial class $K $ to $ \sum_{u \in I/I \cap tIt^{-1}} utK$ (and extended $B$-equivariantly to $\Z[G/K]$).

The main result of the paper (which generalizes \cite[Lemma 3.3]{BBJ18}) is:
\begin{theorem}[Seed relation]\label{uoproothecke}\label{uoprootheckeseed}
The operator $u_{{\varpi^{\mu}}} \in \mathds{U}$ is a right root of the Hecke polynomial $H_{\G,{\fc}}$ in $\End_{\Z[P]}\Z[q^{\pm 1}][G/K]$.
\end{theorem}
\begin{remark}
The minimal polynomial of $u_{{\varpi^{\mu}}}$ has actually its coefficients in the integral Hecke algebra $\End_{\Z[G]}\Z[G/K]$.
\end{remark}
\begin{remark}
This relation has another application; in \cite{Unitarynormrelations2020} we construct a tame norm compatible system of special cycles in a (product of) unitary Shimura variety.
\end{remark}
\begin{remark}
A very interesting and surprising aspect of this work is that in order to establish formulas relating the two non-commuting commutative subrings, $\mathbb{U}$
and $\cH_K(G)$, of the Hecke algebra $\End_{\Z[B]}(\Z[q^{\pm 1}][G\sslash K])$ one has to embed them both in yet another noncommutative ring (the Iwahori--Hecke algebra $\cH_I(\Z[q^{-1}])$), where they actually do commute!
\end{remark}
\subsection{Acknowledgements}
This work is based on results proved in the author’s EPFL 2019 thesis. Therefore, I am very thankful to my adviser D. Jetchev. 
I am also indebted and very grateful to C. Cornut for his comments and reading. 
I would also like to thank T. Wedhorn; the discerning reader will no doubt notice the importance of his paper \cite{wedhorn00}.

\section{Langlands dual group}
Let $\Gamma_{un}=\Gal(F^{un}/F)\simeq \Gal(\overline{k}_F/k_F)$. As before, we let $\sigma \in \Gamma_{un}$ be the arithmetic Frobenius of $F$. The group $\G$ split over $F^{un}$ \cite[XXVI 7.15]{sga3.3}. We consider a Langlands dual group of $\G$ with respect to $\Gamma_{un}$. This group sits in the following short exact sequence
$$\begin{tikzcd} 1 \arrow{r}& \widehat{\G} \arrow{r}& {}^L\G \arrow{r}&\Gamma_{un} \arrow{r}&1,\end{tikzcd}$$
and every choice of \'{e}pinglage $(\widehat{\B},\widehat{\T}, (e_\alpha))$\footnote{Here, for each simple root $\alpha$ of $\widehat{\T}$, $e_\alpha$ is a nonzero element of the root vector space $\text{Lie}(\widehat{\G})_\alpha$.} yields a splitting of the above exact sequence. We fix a $\Gamma_{un}$-invariant \'{e}pinglage \cite[\S1]{Kot84} thus ${}^L\G= \widehat{\G}\rtimes \Gamma_{un}$.

The $\Gamma_{un}$-equivariant isomorphism $X_*(\T)\simeq X^*(\widehat{\T})$ induces a canonical identification between the $\Gamma_{un}$-groups $W(\G_{\overline{F}},\T)$ and the Weyl group $W(\widehat{\G},\widehat{\T})$ and an identification between the $X_*(\Sbf)=X_*(\T)_F$ and $X^*(\widehat{\Sbf})$. 
The inclusion $\Sbf \hra \T$ gives an embedding $X_*(\Sbf)\hra X_*(\T)$ ,which yields a short exact sequence 
$$\begin{tikzcd} 1 \arrow{r}& \widehat{\T}^{1-\sigma}  \arrow{r}& \widehat{\T}  \arrow{r}& \widehat{\Sbf}  \arrow{r}&1  \end{tikzcd},$$
showing that $\widehat{\Sbf}\simeq \widehat{\T}/ (1-\sigma)\widehat{\T}$. Therefore,
$$\widehat{\T}=\Spec(\C[X^*(\widehat{\T})])=\Spec(\C[X_*({\T})]),$$
$$\widehat{\Sbf}=\Spec(\C[X_*({\Sbf})])= \Spec(\C[\Lambda_T])=\Spec\left(\cC_c(\T(F)\sslash \mathcal{T}(\cO_F),\C )\right).$$
In particular, $\widehat{\Sbf}(\C)=\Hom(X_*(\T)_F, \C^\times)$. The above fixed canonical identification $W({\G}_{\overline{F}},{\T})\simeq W(\widehat{\G},\widehat{\T})$, lets $W({\G},{\Sbf})$ operates on $\widehat{\Sbf}$ by duality. The space $\widehat{\Sbf}/{W(\G,\Sbf)}$ has the structure of a smooth affine $\C$-scheme whose coordinate ring is $\C[X_*(\Sbf)]^{W(\G,\Sbf)}$:
$$\widehat{\Sbf}/{W(\G,\Sbf)}=\Spec\left(\C[X_*(\Sbf)]^{W(\G,\Sbf)}\right)=\Spec\left(\C[\Lambda_T]^{W(\G,\Sbf)}\right).$$
Using the untwisted Satake isomorphism of \cite[Theorem \ref{twistedsatakeisomorphism}]{UoperatorsI2021} we obtain
\begin{equation}\label{spechecke}\widehat{\Sbf}/{W(\G,\Sbf)}=\Spec\left(\cH_K(\C)\right).\end{equation}
\section{Unramified representations and unramified $L$-parameters}
Let $\mathcal{W}_F\subset \Gamma_{un}$ whose elements induce an integral power of the Frobenius automorphism $\sigma\colon x \mapsto x^q$ on the algebraic closure of the residue field. 
The valuation $\text{val}\colon \mathcal{W}_F \to \Z$ sends an element $\psi\in  \mathcal{W}_F$ to the power of $\sigma$ it induces, e.g $\text{val}(\sigma)=1$. Define the "Weyl form" of the Langlands group to be ${}^L_w\G:=\widehat{\G}\rtimes \mathcal{W}_F\subset {}^L\G$. 
The isomorphism $\Z \to \mathcal{W}_F$ given by $1 \mapsto \sigma$ defines a semidirect product $\widehat{\G}\rtimes \Z$ and we get a homomorphism $${}^L_w\G \to \widehat{\G}\rtimes \Z.$$
\begin{definition}
An unramified $L$-parameter is a homomorphism $\phi\colon \mathcal{W}_F\to {}^L_w\G$ that verifies the following properties:
\begin{enumerate}[nosep]
\item The composition$\begin{tikzcd} \mathcal{W}_F \arrow{r}{\phi}& {}^L_w\G \arrow[r,twoheadrightarrow] & \mathcal{W}_F\end{tikzcd}$ is the identity.
\item For any $w\in \mathcal{W}_F$, $\phi(w)$ is semisimple.
\item The composition$\begin{tikzcd}\mathcal{W}_F \arrow{r}{\phi}& {}^L_w\G \arrow[r] & \widehat{\G}\rtimes \Z \end{tikzcd}$ factors through $\text{val}$.
\end{enumerate}
Set $\Phi_{un}(\G)\nomenclature[D]{$\Phi_{un}(\G)$}{Equivalence classes of unramified $L$-parameters}$ for the set of equivalence\footnote{Two $L$-parameters are equivalent if they are $\widehat{\G}(\C)$-conjugate.} classes of unramified $L$-parameters.
\end{definition}
The set of $L$-parameters is in bijection with the set of semisimple elements of the form $g\rtimes \sigma \in {}^L\G$. Therefore, $\Phi_{un}(\G)$ identifies with the set of semisimple elements of $\widehat{\G}$ modulo $\sigma$-conjugation.
\begin{definition}
An unramified representation of $\G(F)$ is a homomorphism of groups $\pi\colon \G(F) \to \GL(V)$ where $V$ is a $\C$-vector space verifying the following conditions:
\begin{enumerate}[nosep]
\item $\pi$ is irreducible.
\item The stabilizer of any vector $v \in V$ is an open subgroups of $\G(F)$.
\item For any open subgroup $O\subset \G(F)$, the vector subspace $V^O$ of $O$-fixed vectors is finite dimensional.
\item The subspace $V^K$ is nonzero.
\end{enumerate}
Set $\Pi_{un}(\G)\nomenclature[D]{$\Pi_{un}(\G)$}{Equivalence classes of unramified representations of $\G(F)$}$ for the set of equivalence\footnote{Two representations $(\pi_1,V_1)$ and $(\pi_2,V_2)$ are equivalent if there exists an isomorphism $V_1\to V_2$ sending $\pi_1$ to $\pi_2$.} classes of unramified representations of $\G(F)$.
\end{definition}
\begin{proposition}\label{replparameter}
There is a natural bijection 
$$\Phi_{un}(\G) \simeq \widehat{\Sbf} (\C) / W(\G,\Sbf)  \simeq \Pi_{un}(\G).$$
\end{proposition}
\begin{proof}
In the proof of \cite[Proposition 1.12.1]{BR94}, one shows first the above proposition for the torus $\T$:
$$\Phi_{un}(\T) \simeq \widehat{\Sbf}(\C)\simeq   \Pi_{un}(\T),$$
then deduce it for $\G$ using \cite[Proposition 6.7]{Bor79}.
\end{proof}
Combining Proposition \ref{replparameter} and (\ref{spechecke}) yields
\begin{equation}\label{heckespec}\Phi_{un}(\G)\simeq\Spec(\cH_K(\C)).\end{equation}
\begin{remark}
The above proposition gives an alternative characterization of the untwisted Satake homomorphism. Consider the following injective homormophism
$$\begin{tikzcd}[row sep= small]\cH_K(\C)\arrow{r}& \big\{\Pi_{un}(\G) \to \C\big\}\\
h_g ={\bf 1}_{KgK}\arrow[r, mapsto] & (\pi \mapsto \Tr(\pi(h_g)|_{V^{K}})),
\end{tikzcd}$$
where, $V$ is given a structure of a left $\cH_K(\C)$-module defined by $f\cdot v$ for $f\in \cH_K(\C)$ and $v\in V$ by the formula
$$f\cdot v=\int_G f(g)(\pi(g)\cdot v)d\mu_K(g).$$
By Proposition \ref{replparameter} we get the following commutative diagram 
$$\begin{tikzcd}\cH_K(\C) \arrow{d}{\simeq}\arrow[hook]{rr}{\mathcal{S}_T^G}&&\cC_c(\Lambda_T,\C) \arrow{d}{\simeq}\\
\C[\Pi_{un}(\G)]\arrow{r}{\simeq}& \C[\Pi_{un}(\T)]^{W(\G,\Sbf)}\arrow[r, hook]  & \C[\Pi_{un}(\T)].\qedhere
\end{tikzcd}$$
\end{remark}
\section{The Hecke polynomial}
Let ${\fc} \in \mathcal{M}(\overline{F})\nomenclature[D]{$\fc$}{Fixed conjugacy class in $\mathcal{M}(\overline{F})$}$ and $\mu_{\fc} \in X_*(\T)$ be the unique $\B_{\overline{F}}$-dominant cocharacter of $\T_{\overline{F}}$. 
Both, ${\fc}$ and $\mu_{\fc}$ have the same field of definition, a finite unramified extension $F({\fc})\subset F^{un}\nomenclature[D]{$F(\fc)$}{Field of definition of $\fc$}$ of $F$. 
Set $d=[{F(\fc)}:F]\nomenclature[D]{$d$}{$[{F(\fc)}:F]$}$  
and let $$\Norm_{F(\fc)/F} \fc:=[\prod_{\tau\in \Gal(F(\fc)/F)}\tau(\mu_{\fc})]\in \mathcal{M}(F)$$ be the norm of $\fc$\footnote{It is straightforward that the conjugacy class $\Norm_{F(\fc)/F} \fc$ does not depend on the choice of the representative $\mu_\fc$.}. 
We may assume that for some representative of the conjugacy class $\Norm_{F(\fc)/F} {\fc}$ takes values in the torus $\T$ (and hence for all). 
The conjugacy class ${\fc}\in\mathcal{M}({F(\fc)})$ determines a Weyl orbit of a character of $\widehat{\T}$, in which there is a unique $\widehat{\mu}_\fc \in X^*(\widehat{\T})$ that is dominant with respect to the Borel subgroup $\widehat{\B}$.

Let $(r_{{\fc}},V)$ be a representation of $^{L}(\G_{F(\fc)})$ (unique up to isomorphism) satisfying the conditions:
\begin{itemize}[nosep]
\item The restriction of $r_{{\fc}}$ to $\widehat{\G}$ is irreducible with highest weight $\widehat{\mu}_\fc$.
\item For any admissible invariant splitting of $^{L}(\G_{F(\fc)})$ the subgroup $\Gamma_{un}^d$ of $^{L}(\G_{F(\fc)})$ acts trivially on the highest weight space of $r_{{\fc}}$.
\end{itemize}
Fix an invariant admissible splitting $^{L}(\G_{F(\fc)})=    \widehat{\G}\rtimes \Gamma_{un}^d$.
\begin{definition}[The Hecke polynomial]\label{defheckepol}
For every $\widehat{g} \in \widehat{\G}$, consider the following polynomial:
$$
P_{\G,{\fc}}(X)=\det \left(X-q^{d \langle \mu_{\fc},\rho\rangle}r_{{\fc}}\left(( \widehat{g}\rtimes \sigma)^d\right)\right).
$$
By varying $\widehat{g}$, the coefficients of $P_{\G,{\fc}}$ are viewed as elements of the algebra of regular functions of $\Phi_{un}(\G)$.
Let $H_{\G,{\fc}}\in \cH_K(\C)[X]\nomenclature[D]{$H_{\G,{\fc}}$}{The Hecke polynomial attached to the pair $(\G,\fc)$ in $\cH_K(\C)[X]$}$ be the Hecke polynomial corresponding to $P_{\G,{\fc}}$ via (\ref{heckespec}) (compare with \cite[\S 6]{BR94}).
\end{definition}

\section{Explicit twisted Satake transform}

Let $\mu\in \Norm_{F(\fc)/F}{\fc}$ be the cocharacter of $\T$ which is $\B$-dominant, i.e. $\varpi^\mu$ is antidominant. Let $\mathbf{L}$ be the centralizer of $\mu$ in $\G$. Let $\P$ be the largest parabolic subgroup of $\G$ relative to which $\mu$ is dominant, $\mathbf{L}$ is a Levi factor of $\P$ and $\U_P^+$ the unipotent radical of $\P$. By definition we have $\T\subset \mathbf{L} $ and $\U_P^+ \subset \U^+$. 
Set $K_\dag=\dag \cap K$ for any $\dag\in \{P,L ,U_P^+\}$.
Denote by $f_{[\mu]}= {\bf 1}_{K\varpi^{\mu}K}\in\cH_K(R)$ (resp. $g_{[\mu]}={\bf 1}_{ \varpi^{\mu} K_{L}}\in \cC_c(L\sslash K_{L},R)$, resp. $h_{[\mu]}={\bf 1}_{\varpi^\mu T_1}\in \cC_c(T\sslash T_1,R)\simeq R[\Lambda_T]$ the characteristic function of the double coset corresponding to $[\mu]$.  Let $p\colon \G_{sc}\to \G$ be the simply connected covering of the derived group of $G$ and let $\Sbf_{sc}$ be the unique maximal $F$-split torus of $\G_{sc}$ such that $p(\Sbf_{sc})\subset \Sbf$. The map $p$ defines a homomorphism from $X_*(\Sbf_{sc})$ to $X_*(\Sbf)$. We are interested in the set
$$\Sigma_F(\mu)=\{\nu \in X_*(\Sbf): \mu-\nu \in \text{Im}(X_*(\Sbf_{sc})) \text{ and } w\nu \preceq \mu\text{ for all } w\in W(\G,\Sbf) \}.$$
\begin{remark}\label{saturatedsets}
The above $W$-invariant sets of weights plays a prominent role in representation theory and they are called "saturated sets of weights". Moreover, we have (see \cite[\S 2.3]{Ko1}, \cite[13.4 Exercise]{Humphreys72} and Bourbaki's \cite[Chapter VI, Exercises of \S 1 and \S2]{BourbakiLieV}) that
$$\Sigma_F(\mu)=\bigsqcup_{\lambda \in X_*(\Sbf)\cap \overline{\cC}\colon \lambda \preceq \mu}W\lambda$$
where $\preceq$ denotes\footnote{Compare with \cite[Definition \ref{presimorder}]{UoperatorsI2021}} the partial order on $X_*(\Sbf)\cap \overline{\cC}$ defined by
$$\lambda \preceq \nu \Leftrightarrow \nu-\lambda =\sum n_\alpha \alpha^\vee, n_\alpha \in \Z_{\ge 0}.\qedhere$$
Moreover, when $\mu$ is minuscule then $\Sigma_F(\mu)=W(\G,\Sbf) \mu$ \cite[Remark \ref{genkotwitz}]{UoperatorsI2021}.
\end{remark}
We have the following explicit description of the twisted Satake homomorphism
\begin{proposition}\label{satake}
Write
$$\dot{\mathcal{S}}_T^G(f_{[\mu]})=\sum_{\nu \in \Sigma_F(\mu)}c(\nu).{\bf 1}_{\varpi^{\nu}T_1} \in \cC_c(T\sslash T_1,\Z),$$
and the coefficients $\{c(\nu)\}$ are positive powers of $q$ and verifies
$$c(w\nu)=q^{\langle\delta,\nu-w(\nu)\rangle}c(\nu)\text{ for all }w\in W(\G,\Sbf), \text{ with $c(\mu)=1$.}$$ 
\end{proposition}
\begin{proof}
This is a particular case of \cite[Theorem \ref{twistedsatakeisomorphism} \& Theorem \ref{positivitythm}]{UoperatorsI2021}. 
The twisted Satake isomorphism ensures that $\dot{\mathcal{S}}_T^G(f_{[\mu]}) \in \cC_c(T\sslash T_1,\Z)^{\dot{
W}}$ where $\dot{W}$ denotes the Weyl group with its twisted dot-action (See \cite[\S \ref{weyldotsect}]{UoperatorsI2021}). 
This shows that $c(\nu)q^{\langle \delta,\nu \rangle}=c(w(\mu))q^{\langle \delta,w(\nu) \rangle}$ for all $w \in W(\G,\Sbf)$. The coefficient $c(\mu)=1$ is obtained by \cite[Lemma 2.3.7 (b)]{Ko1} using \cite[Remark \ref{untwistedtotwisted}]{UoperatorsI2021}. 

The fact that $c(\nu)>0$ if and only $\nu\in \Sigma_F(\mu)$ is well known in this unramified case; it follows by \cite[Lemma 2.3.7 (a)]{Ko1} for the "only if" and \cite{Ra} for the "if".
\end{proof}
\section{Seed relations and $\mathds{U}$-operators}
{Using the fixed \'{e}pinglage, we can consider a $\Gamma_{un}$
-equivariant embedding ${}^L\T= \widehat{\T} \rtimes \Gamma_{un}\hra {}^L\G$. The composition 
$$\begin{tikzcd}{}^L(\T_{F(\fc)}) \arrow[r,hook]& {}^L(\G_{F(\fc)}) \arrow[hookrightarrow]{r}{r_{{\fc}}}& \GL(V)\arrow{r}{P_{\G,{\fc}}}&\C[X]\end{tikzcd},$$
is independent of all fixed choices. The restriction of $r_{{\fc}}$ to $\widehat{\T}$ yields a weight space decomposition 
$$V=\bigoplus_{\lambda \in \Sigma_E(\mu_{\fc})}V_{\widehat{\lambda}}.$$
We have 
$$\mathcal{S}_T^G(P_{\G,{\fc}})=\det \left(X-q^{d \langle \mu_{\fc},\rho\rangle}r_{{\fc}}|_{{}^L(\T_{F(\fc)})}\left(( \widehat{t}\rtimes \sigma)^d\right)\right)\in \C[\Phi_{un}(\T)]^{W(\G,\Sbf)}.$$
Define the twisted restriction of $r_{{\fc}}$ to be the morphism of schemes
$$\begin{tikzcd}[column sep=small]r_T\colon{}^L(\T_{F(\fc)})= \widehat{\T}\rtimes \Gamma_{un}^d\arrow{r}&\GL(V) \end{tikzcd}$$
given on $\C$-points by
\begin{equation}\label{twistedrepresentation}r_T(1\rtimes \sigma^d)=r_{{\fc}}(1\rtimes \sigma^d)\text{ and } r_T(\widehat{t}\rtimes 1) \cdot v_\lambda= q^{-\langle \rho, \lambda \rangle}\lambda(\widehat{t})\cdot v_\lambda\end{equation}
for $v_\lambda\in V_\lambda$ for all $\lambda\in \Sigma(\mu_{\fc})$. The homomorphism $r_T$ is not a homomorphism of
groups but maps conjugacy classes to conjugacy classes and it is defined to ensure, using \cite[Remark \ref{untwistedtotwisted}]{UoperatorsI2021} and (\ref{twistedrepresentation}), that
\begin{align*}\dot{\mathcal{S}}_T^G(P_{\G,{\fc}})&=\eta_B\circ  {\mathcal{S}}_T^G(P_{\G,{\fc}})\\
&=\det \left(X-q^{-d \langle \mu_{\fc},\rho\rangle}r_T\left(( \widehat{t}\rtimes \sigma)^d\right)\right)\in \C[\Phi_{un}(\T)].
\end{align*}
\begin{remark}{Note that our choice of the twisted representation $r_T$ depends crucially on the normalization of the isomorphism $X_*(\Sbf) \simeq \Lambda_T$. We have adopted the following isomorphism $\lambda \mapsto \varpi^{\lambda}$.  
Using \cite[Remark \ref{untwistedtotwisted}]{UoperatorsI2021} and $\delta_B({\varpi^{\lambda}})^{1/2}=q^{-\langle  \lambda, \rho \rangle}$, we see that
$$
\begin{tikzcd}[column sep= large, row sep= large]
X_*(\Sbf)\otimes_\Z\C\arrow{d}{\lambda \mapsto \varpi^{\lambda}T_1} [swap]{\simeq}\arrow{rr}{\eta\colon \lambda \mapsto q^{-\langle \lambda, \rho \rangle }}&&X_*(\Sbf)\otimes_\Z\C\arrow{d}{\simeq}\\
\Lambda_T \otimes_\Z\C \arrow{rr}{\eta\colon tT_1 \mapsto \delta(t)^{1/2} tT_1}&&\Lambda_T\otimes_\Z\C.\qedhere
\end{tikzcd}$$
}\end{remark}
As opposed to \cite[Proposition 2.7]{wedhorn00}, we insist on the fact that we do not assume $\mu$ to be minuscule in the following proposition. 
\begin{proposition} \label{propwedhorn}
\begin{enumerate}
\item Let $\Sbf^{F(\fc)}\subset \T$ denotes the maximal split torus of $\G_{F(\fc)}$ containing the image of $\mu_{\fc}$, let $\overline{\cC}_{F(\fc)}\subset \cB(\G_{F(\fc)},{F(\fc)})_{\ext}$ be the closed vectorial chamber corresponding to the Borel $\B_{F(\fc)}$. We have
$$\text{deg}(H_{\G, {\fc}})\ge \sum_{\lambda \in X_*(\Sbf^{F(\fc)})\cap \overline{\cC}_{F(\fc)}\colon \lambda \preceq \mu_{\fc}} \# (W({\G}_{},\Sbf^{F(\fc)}) {\lambda})=\#\Sigma_{F(\fc)}({\mu}_{\fc})$$
\item The twisted restriction $r_T$ of $r_{{\fc}}$ to ${}^L(\T_{F(\fc)})$ is isomorphic to a direct sum 
$$V=\bigoplus_{\Sigma_{F(\fc)}(\mu_{\fc})}V_{\widehat{\lambda}}$$
where, $V_{w(\widehat{\mu})}$ is one-dimensional with generator $v_{\widehat{\lambda}}$ for any $w \in W$, such that
\begin{equation}r_T(\widehat{t}\rtimes \sigma^d) \cdot v_{\sigma^{d(r-1)}w(\widehat{\mu})}=q^{-\langle \rho, w({\mu}) \rangle}  w(\widehat{\mu}) (\widehat{t})\cdot v_{w(\widehat{\mu})}.\end{equation}
\end{enumerate}
\end{proposition}
\proof We will just imitate the proof of \cite[(2) Proposition 2.7]{wedhorn00} but without requiring $\mu$ to be minuscule.
\begin{enumerate}
\item 

Fix a  Borel pair $(\widehat{\T},\widehat{\B})$ of $\widehat{\G}$ and let $\widehat{\mu}_{\fc}$ be the dominant character of $\widehat{\T}$ corresponding to the conjugacy classe ${\fc}$. By definition of the Hecke polynomial, its degree is the dimension of the representation $r_{{\fc}}$ which is irreducible with highest weight $\widehat{\mu}_{\fc}$ as a representation of $\widehat{\G}$. By remark \ref{saturatedsets}, the only weights of $r_{{\fc}}$ are the elements $\bigsqcup_{\widehat{\lambda}} W(\widehat{\G},\widehat{\T}) \widehat{\lambda}$ where the disjoint union is taken over dominant wights $\widehat{\lambda}\preceq \widehat{\mu}_{\fc}$ (here $\preceq$ is the usual partial order on dominant weights $X^*(\widehat{\T})^{\text{dom}}$). By definition of the dual group, we then have
\begin{align*}\bigsqcup_{{\widehat{\lambda}\in X^*(\widehat{\T})^{\text{dom}}\colon \widehat{\lambda}\preceq \widehat{\mu}_{\fc}}} W(\widehat{\G},\widehat{\T}) \widehat{\lambda}&= \bigsqcup_{\lambda \in X_*(\Sbf^{F(\fc)})\cap \overline{\cC}_{F(\fc)}\colon \lambda \preceq \mu_{\fc}} W({\G}_{{F(\fc)}},\Sbf^{F(\fc)}) {\lambda}\\ &=\Sigma_{F(\fc)}({\mu}_{\fc}).\end{align*}
\item The twisted restriction $r_T$ of $r_{{\fc}}$ to ${}^L(\T_{F(\fc)})$ is isomorphic to a direct sum 
$$V=\bigsqcup_{{\widehat{\lambda}\in X^*(\widehat{\T})^{\text{dom}}\colon \widehat{\lambda}\preceq \widehat{\mu}_{\fc}}} V_{\widehat{\lambda}}$$
and the highest weight space $V_{\widehat{\mu}_{\fc}}$ is one-dimensional\footnote{The weight spaces in the weyl orbit of the highest weight are one dimensional, but outside this distinguished weyl orbit, there are weight spaces which are not $1$ dimensional.} with generator $v_{\widehat{\mu}_{\fc}}$. 
Accordingly, $V_{\widehat{\lambda}}$ is one-dimensional for any $\widehat{\lambda}\in W(\widehat{\G},\widehat{\T})  \widehat{\mu}_{\fc}$. 
The conjugacy class ${\fc}$ being defined over ${F(\fc)}$, we see that $\langle \sigma^{n}\rangle $ stabilizes $W(\widehat{\G},\widehat{\T})  \widehat{\mu}_{\fc}$. 

Choose for each classe $Z\in W(\widehat{\G},\widehat{\T})  \widehat{\mu}_{\fc}/\langle \sigma^{d}\rangle$ a representative $\widehat{\lambda}_Z \in Z$ and a vector $v_{\widehat{\lambda}_Z}\in V_{\widehat{\lambda}_Z}$. Define
$$v_{\sigma^{rd}(\widehat{\lambda}_Z)}:=r_{{\fc}}(1\rtimes \sigma^{rd})\cdot v_{\widehat{\lambda}_Z},\quad \text{for }1\le r < r_Z:=\min\{s \colon \sigma^{sd}\widehat{\lambda}_Z=\widehat{\lambda}_Z\}.$$
Therefore, taking $r=-1$ gives
\begin{align*}r_T(\widehat{t}\rtimes \sigma^d) \cdot v_{\sigma^{d(r-1)}(\widehat{\lambda}_Z)}
&= r_T(\widehat{t}\rtimes 1) \cdot v_{\widehat{\lambda}_Z}
\overset{(\ref{twistedrepresentation})}{=}q^{-\langle \rho, \lambda \rangle}  \widehat{\lambda}_Z (\widehat{t})\cdot v_{\widehat{\lambda}_Z}.\qedhere
\end{align*}
\end{enumerate} 
}
{\begin{lemma}\label{annihilationinphiT}
We have $(\dot{\mathcal{S}}_T^G H_{\G,{\fc}})(\mu)=0$ in $\cC_c(T\sslash T_1,R)$.
\end{lemma}
\proof
The conjugacy classe $[\mu]$ (resp. ${\fc}$) gave rise to a dominant character $\widehat{\mu}$ (resp. $\widehat{\mu}_{\fc}$) of $\widehat{\T}$ and
$$\widehat{\mu}=\widehat{\mu}_{\fc} \sigma(\widehat{\mu}_{\fc}) \cdots \sigma^{d-1}(\widehat{\mu}_{\fc}).$$
To prove the lemma, it suffices to show that 
$$\det \left(X - q^{d\langle \mu_{\fc} , \rho \rangle} r_T|_{V_{\widehat{\mu}_{\fc}}}\big( (\sigma \ltimes \widehat{t})^d \big)\right)\in \C[\Phi_{un}(\T)][X]$$
has $\widehat{\mu}(\widehat{t})$  as a root for all $\widehat{t}\in \widehat{\T}$. Identify $ \Phi_{un}(\T)$ with the set of $\sigma$-conjugacy classes $\{\widehat{t}\}$ of elements $\widehat{t} \in \widehat{\T}(\C)$. For any $v\in V_{\widehat{\mu}}$, we have
\begin{align*}
q^{d\langle \mu_{\fc} ,\rho\rangle} r_T\big((\sigma \ltimes \widehat{t})^d\big) \cdot v&= q^{d\langle \mu_{\fc} ,\rho\rangle} r_T\big(\sigma^d \ltimes (\widehat{t}\sigma(\widehat{t} )\cdots \sigma^{d-1}(\widehat{t}))\big) \cdot v\\
\overset{\text{Prop. \ref{propwedhorn}}}&{=}\widehat{\mu}_{\fc}\big( \widehat{t}\sigma(\widehat{t}) \cdots \sigma^{d-1}(\widehat{t})\big)\cdot v\\
&= \widehat{\mu}_{\fc}(\widehat{t}) \sigma(\widehat{\mu}_{\fc})(\widehat{t}) \cdots \sigma^{d-1}(\widehat{\mu}_{\fc})(\widehat{t})\cdot v\\
&= \widehat{\mu}(\widehat{t})\cdot v.\qedhere
\end{align*}}
We will show now the main theorem of the paper: 
\begin{theorem}[Seed relation]\label{uoproothecke}
The operator $u_{{\varpi^{\mu}}} \in \mathds{U}$ is a right root of the Hecke polynomial $H_{\G,{\fc}}$ in the non-commutatif $R$-algebra $\text{\emph{End}}_{P}(\cC_c(G/K,R))$.
\end{theorem}

\begin{proof} 
Under the identifications $\Lambda_T \simeq X_*(\T)_F \simeq X^*(\widehat{\T})_F$
the element ${\varpi^{\mu}}T_1 \in \Lambda_T^-$ corresponds to the function $t \mapsto \widehat{\mu}(t)$.

Recall that by \cite[Lemma \ref{Pinvariance}]{UoperatorsII2021} $u_{{\varpi^{\mu}}}\in \End_P\cC_c(G\sslash K,\Z)$ and the coefficients of $H_{\G,{\fc}}$ are in $\cH_K(R)\simeq \End_G\cC_c(G\sslash K,R)$ \cite[2.8]{wedhorn00}, thus 
$$H_{\G,{\fc}}(u_{{\varpi^{\mu}}})\in \End_P\cC_c(G\sslash K,R).$$
Using \cite[Theorem \ref{compatibility}]{UoperatorsI2021}, we see that $\dot{\Theta}_{\text{Bern}}\circ\dot{\mathcal{S}}_T^G H_{\G,{\fc}}\in Z(\cH_I(R))[X]$.
Write $H_{\G,{\fc}} =\sum_{k=1}^r h_k X^k$ and $\bar{h}_k =\dot{\Theta}_{\text{Bern}}\circ\dot{\mathcal{S}}_T^G (h_k)\in Z(\cH_I(R))$. So $\bar{h}_k*_I {\bf 1}_K={\bf 1}_K *_I \bar{h}_k=h_K$. 
We then have for any $p\in P$
\begin{align*}
{\bf 1}_{pK}  \bullet H_{\G,{\fc}}( u_{{\varpi^{\mu}}})
&=\sum_{k=1}^r   ({\bf 1}_{pK} \bullet  u_{{\varpi^{\mu}}}^{k}) *_K {h}_k \\ 
&=\sum_{k=1}^r   ({\bf 1}_{pI} *_I  i_{{\varpi^{\mu}}}^k) *_I {\bf 1}_K *_K{h}_k \\
&=\sum_{k=1}^r   ({\bf 1}_{pI} *_I  i_{{\varpi^{\mu}}}^k) *_I (\frac{1}{[K:I]}{\bf 1}_K *_I {\bf 1}_K *_I \bar {h}_k) \\
&=\sum_{k=1}^r   ({\bf 1}_{pI} *_I  i_{{\varpi^{\mu}}}^k) *_I {\bf 1}_K *_I \bar {h}_k \\
&={\bf 1}_{pI} *_I  \left(\sum_{k=1}^r   i_{{\varpi^{\mu}}}^k *_I \bar{h}_k\right)*_I{\bf 1}_{K} \\
&={\bf 1}_{pI} *_I  \left(\sum_{k=1}^r  \bar{h}_k  *_Ii_{{\varpi^{\mu}}}^k \right)*_I{\bf 1}_{K} \\
&={\bf 1}_{pI} *_I  \left((\dot{\Theta}_{\text{Bern}}\circ\dot{\mathcal{S}}_T^G H_{\G,{\fc}})(i_{{\varpi^{\mu}}}^k)\right)*_I{\bf 1}_{K}\\
&={\bf 1}_{pI} *_I  \dot{\Theta}_{\text{Bern}}\left((\dot{\mathcal{S}}_T^G H_{\G,{\fc}})(\varpi^{\mu}T_1)\right)*_I{\bf 1}_{K}\\
\overset{\text{Lemma \ref{annihilationinphiT}}}&{=}0.
 \end{align*}
We have shown
$
H_{\G,{\fc}} (u_{{\varpi^{\mu}}})= \sum_{k=1}^r  h_k \circ u_{{\varpi^{\mu}}}^k=0 \in \End_{P}(\cC_c(G/K,R)).$
\end{proof}
\begin{remark}
If $\mu_{\fc}$ is minuscule, then $\Sigma_F(\mu_{\fc})=W(\G_{\overline{F}}, \T)\,\mu_{\fc}$ and accordingly the degree of the Hecke polynomial is 
$$\text{\emph{deg}}(H_{\G,{\fc}})=\left|W(\G_{\overline{F}}, \T)\,\mu_{\fc} \right|.$$
In particular, $\text{\emph{deg}}(H_{\G,{\fc}})\ge \text{\emph{deg}}(P_{\mu})  =|W/W_\mu|= \left|W(\G, \Sbf)\,\mu \right|$, where $P_\mu$ is the minimal polynomial of $u_{{\varpi^{\mu}}}$ in $Z(\cH_I(R))$ (see proof of \cite[Theorem \ref{integrality}]{UoperatorsII2021}). Therefore, if $\G$ is a split group, $\mu_{\fc}$ minuscule and $E=F$, then
$$H_{G,[\mu]}=P_{\mu}*_I{\bf 1}_K.\qedhere$$
\end{remark}
\section{{B\"ultel}'s annihilation relation}\label{liftbultel}
{In this last section we will show how Theorem \ref{uoproothecke} lifts (generalizes) a previously known result due to {B\"ultel} \cite[1.2.11]{Bu97}.

Let $\dot{\mathcal{S}}_P\colon \cC_c(P\slash K_{P},\Q)\to \cC_c(L\slash K_{L},\Q)$ be the canonical homomorphism given by
$$f \mapsto \left( m\mapsto \int_{U_P^+}f(nm)d\mu_{U_P^+}(n)\right),$$
where $d\mu_{U_P^+}$ is the left-invariant Haar measure giving $K_{U_P^+}$ volume $1$. Both $\Q$-modules $\cC_c(P\slash K_{P},\Q)$ and $ \cC_c(L\slash K_{L},\Q)$ are actually $\Q$-algebras (by \cite[Lemma \ref{actionhecke}]{UoperatorsI2021}) and the transform $\dot{\mathcal{S}}_P$ is an algebra homomorphism. Indeed, let $f,g \in \cC_c(P\slash K_{P},\Q)$ then 
\begin{align*}
\dot{\mathcal{S}}_P(f*_{K_{P}} g)(p)&=\int_{U_P^+} (\int_P f(a)g(a^{-1}up) d\mu_{P}(a)) d\mu_{U_P^+}(u)\\
&=\int_{U_P^+}  \int_L \int_{U_P^+}  f( nm)g(m^{-1}n^{-1} u p) d\mu_{U_P^+}(n)d\mu_{L}(m)  d\mu_{U_P^+}(u)\\
&=\int_{U_P^+} \left( \int_L  f( nm) d\mu_{U_P^+}(n) \right) \left(\int_{U_P^+} g( m^{-1}p u) d\mu_{U_P^+}(u) \right) d\mu_{L}(m)\\
&=\dot{\mathcal{S}}_P(f)*_{K_{P}} \dot{\mathcal{S}}_P(g)(p)
\end{align*}
where, $d\mu_{P}$ denotes the left invariant Haar measure giving $K_{P}$ measure $1$.

We also consider the map $|_P$ sending any function on $G$ to its restriction to $P$. Using the Iwasawa decomposition $G=PK$ (\cite[{Proposition \ref{Iwasawadec}}]{UoperatorsI2021}) one shows that this is actually an algebra homomorphism
$$\begin{tikzcd} {|}_P\colon \cH_K(R)\arrow{r}& \cC_c(P\sslash K_{P} ,R),\end{tikzcd}$$
and a ${|}_P$-linear module homomorphism
$$\begin{tikzcd} {|}_P\colon \cC_c(G/K,R)\arrow{r} &\cC_c(P/ K_{P} ,R).\end{tikzcd}$$
\begin{lemma}\label{reduction} Let $p \in P$ and $m \in L$, then: $${\bf 1}_{pK}|_P={\bf 1}_{p K_{P}} \text{ and } \dot{\mathcal{S}}^P_L({\bf 1}_{m K_{P}})=|m K_{U_P^+}m^{-1}|_{U_P^+} {\bf 1}_{mK_L}.$$
\end{lemma}
\proof The first equality is a direct consequence of the Iwasawa decomposition. For the second it is deduced from the fact that $K_{P}= K_{L} . K_{U_P^+}$ given in \cite[Proposition \ref{Iwasawadec}]{UoperatorsI2021}:
\begin{align*}\dot{\mathcal{S}}^P_L( {\bf 1}_{mK_{P}})(a)&=\int_{U_P^+}{\bf 1}_{mK_{P}}(ua)d\mu_{U_P^+}(u).
\end{align*}
The integrand is nonzero if and only if $u a \in m K_{P}= m K_{L} \cdot  K_{U_P^+}$, but since ${L} \cap  {U_P^+}=\{1\}$, we have  
$$ u \in aK_{U_P^+}a^{-1} \text{ and }a \in m K_{L},$$
which is equivalent to $u \in m K_{U_P^+}m^{-1}$ and $w\in  m K_{L}$. Therefore,
\begin{align*}
\dot{\mathcal{S}}^P_L( {\bf 1}_{mK_{P}})&= |m K_{U_P^+}m^{-1}|_{U_P^+}{\bf 1}_{m K_{L}}.\qedhere
\end{align*}

Observe that if $m K_{U_P^+}m^{-1} \subset K_{U_P^+}$ then
$$|m K_{U_P^+}m^{-1}|_{U_P^+}=\frac{1}{[K_{U_P^+}:m K_{U_P^+}m^{-1}]}=\frac{1}{[K_{P}:m K_{P}m^{-1}]}.$$
\begin{lemma}
We have a following commutative diagram of $R$-algebras
$$
\begin{tikzcd}[column sep=normal,row sep=large]
\cH_K(R)\arrow{d}[swap]{\dot{\mathcal{S}}_T^G}{\simeq} \arrow[hookrightarrow]{rr}{\dot{\mathcal{S}}_L^G}&&\cC_c(L\sslash K_{L},R)\arrow{d}{\dot{\mathcal{S}}_T^L}[swap]{\simeq} \\\hspace*{\fill}
R[\Lambda_T]^{\dot{W}}\arrow[hookrightarrow]{rr}&& R[\Lambda_T]^{\dot{W}_L}
\end{tikzcd}$$
where, $W_L$ denotes the relative Weyl group of $L$ (which is equal to the subgroup $W_\mu$ of elements in $W$ fixing $\mu$). The lowest horizontal arrow is the inclusion of $W$-invariants into $W_L$-invariants.
\end{lemma}
{\begin{proof}
By definition of the parabolic $P$, multiplication in $G$ gives a bijection
\begin{equation}\label{mupmumu}(U^+\cap L)\cdot U_P^+ \iso U^+.\end{equation}
For any $m\in L$ and $h\in \cH_K(R)$
\begin{align*}\dot{\mathcal{S}}_T^G(h)(m)&= \int_{U^+}  h(um)d\mu_{U^+}(u)&\text{\cite[Lemma \ref{explicitsatake1}]{UoperatorsI2021}}\\
&= \int_{U_P^+} \int_{U^+\cap L}   h(u_1u_2m) d\mu_{U_P^+}(u_1)  \, d\mu_{U^+\cap L}(u_2)&\text{by (\ref{mupmumu})}\\
&=\int_{U^+\cap L} \left(  \int_{U_P^+} h(u_1u_2m) d\mu_{U_P^+}(u_1)  \right) d\mu_{U^+\cap L}(u_2)\\
&=\int_{U^+\cap L} \dot{\mathcal{S}}_L^G(h)(u_2m) d\mu_{U^+\cap L}(u_2)\\
&=\dot{\mathcal{S}}_T^L \circ \dot{\mathcal{S}}_L^G(h)(m).
\end{align*}
Therefore, $\dot{\mathcal{S}}_T^G=\dot{\mathcal{S}}_T^L \circ \dot{\mathcal{S}}_L^G$ which confirms the claimed commutativity of the above diagram. Finally, the vertical maps are isomorphisms by \cite[Theorem \ref{twistedsatakeisomorphism}]{UoperatorsI2021}.
\end{proof}}
Let us reformulate the above twisted Satake homomorphism $\dot{\mathcal{S}}_L^G$ as a homomorphism of endomorphism rings. We have a commutative diagram:
$$\begin{tikzcd}[column sep= small]\cH_K(R)\arrow[d, equal]\arrow{r}{|_P}& \cC_c(P\sslash  K_{P},R)\arrow[d, equal] \arrow[hookrightarrow]{rr}{\dot{\mathcal{S}}_P}&& \cC_c(L\sslash  K_{L},R)\arrow[d,equal]\\
\End_G\cC_c(G/   K,R)\arrow[hook]{r}{(1)}&\End_P\cC_c(G/   K,R)\arrow{r}{(2)}&\End_L\cC_c(L/K_{L},R)\arrow{r}{(3)}[swap]{\simeq}&\End_L\cC_c(L/K_{L},R).
\end{tikzcd}$$
Let us first say few words about the homomorphisms $(1)$ and $(2)$:
\begin{itemize}[nosep]
\item We have used the Iwasawa decomposition $G=PK$ to identify $G/K\simeq P/K_{P}$ for the middle vertical arrow, accordinly the homomorphism $|_P$ induces the canonical injection $(1)$:
$$\begin{tikzcd}
\End_G\cC_c(G/   K,R)\arrow[r,hook]&\End_P\cC_c(G/   K,R).\end{tikzcd}$$
\item We have a homomorphism of rings
$$ \begin{tikzcd}[row sep=small]\End_P\cC_c(G/ K,R)\arrow{r}& \End_P\cC_c(U_P^+ \backslash G/ K,R)\\
f \arrow[r,mapsto] & (U_P^+gK \mapsto \Pi(f(gK)))
\end{tikzcd}$$
where $\Pi$ is the natural obvious map $R[G/ K]\to R[U_P^+ \backslash G/ K]$. But since $P=LU_P^+$, we actually have $\End_P\cC_c(U_P^+ \backslash G/ K,R)=\End_L\cC_c(U_P^+ \backslash G/ K,R)$.

Using the Iwasawa decomposition again $G= U_P^+LK$, we get a bijection
$$U_P^+ \backslash G/K \simeq L/K_{L}.$$
Thus, the homomorphism $(2)$ is the composition $$\begin{tikzcd}\End_P\cC_c(G/ K,R)\arrow{r}&\End_L\cC_c(U_P^+ \backslash G/ K,R)\arrow{r}{\simeq}& \End_L\cC_c(L/ K_{L},R).
\end{tikzcd} $$
\item The homomorphism $(3)$ is the twist by the modulus function $\delta$.
\end{itemize}
\begin{lemma}\label{thetatogmu}
The operator $u_{{\varpi^{\mu}}}$ lives in $\End_P\cC_c(G/ K,R)$ and its image by the composition $(3) \circ (2)$ is precisely $g_{[\mu]}$.
\end{lemma}
\proof Let us first compute the image of the operator $u_{{\varpi^{\mu}}}$ by the map $(2)$. We have for all $a\in L$ (see \cite[Lemma \ref{Pinvariance}]{UoperatorsII2021})
\begin{align*}u_{{\varpi^{\mu}}}({\bf 1}_{U_P^+aK})&=\sum_{p^\prime  \in [U_P^+\cap I^+/ U_P^+\cap {\varpi^{\mu}}I^+{\varpi^{-\mu}}]} \mathbf{1}_{U_P^+a p^\prime {\varpi^{\mu}} K}\\
&=\#(U_P^+\cap I^+/ U_P^+\cap {\varpi^{\mu}}I^+{\varpi^{-\mu}}) \mathbf{1}_{U_P^+  a{\varpi^{\mu}} K}\\
&=\#(I^+/  {\varpi^{\mu}}I^+{\varpi^{-\mu}})  \mathbf{1}_{U_P^+  a{\varpi^{\mu}} K}& \text{\cite[Lemma \ref{bijection42}]{UoperatorsII2021}}
\end{align*}
Hence, the image of $u_{{\varpi^{\mu}}}\in \End_P\cC_c(G/K,R)$ by $(2)$ is $$\#(I^+/  {\varpi^{\mu}}I^+{\varpi^{-\mu}}) g_{[\mu]}= \delta_B({\varpi^{-\mu}}) g_{[\mu]}=q^{2\langle \mu,\rho \rangle } g_{[\mu]}.$$ 
Finally, $(3)$ shows that the image of $u_{{\varpi^{\mu}}} $ by the composition $(3) \circ (2)$ is $g_{[\mu]}\in \End_L\cC_c(L/K_{L},R)$.\qed
 
Bültel's annihilation result we have mentioned earlier is:
\begin{corollary}[Bültel's annihilation]\label{bultel's}
We have
$$\dot{\mathcal{S}}_L^G(H_{\G,{\fc}}(g_{[\mu]}))=0 \in \cC_c(L\sslash  K_{L},R).$$
\end{corollary}
Bültel's result as stated in \cite[\S 2.9]{wedhorn00} requires the conjugacy class $\fc$ to be minuscule. We will derive this corollary from Theorem \ref{uoproothecke}, showing that the assumption "minuscule" is superfluous.

\proof By definition of the "excursion" pairing \cite[\S \ref{pairingsection}]{UoperatorsII2021} and the proof of Lemma \ref{thetatogmu}, we see that for all $p\in P$:
\begin{align*}0\overset{\text{Theorem \ref{uoproothecke}}}&{=}\left(H_{\G,{\fc}}(u_{{\varpi^{\mu}}})\bullet {\bf 1}_{pK}\right)|_P\\
&=  {\bf 1}_{pK_P} *_{K_P}  {\bf 1}_{K_P \varpi^{\mu}K_P}*_{K_P}  (H_{\G,{\fc}})|_P.\end{align*}
This shows that $$(H_{\G,{\fc}})|_P ({\bf 1}_{K_P \varpi^{\mu}K_P})=0,$$ and consequently we conclude
\begin{align*}
\dot{\mathcal{S}}_G^L(H_{\G,{\fc}})(g_{[\mu]})&= \dot{\mathcal{S}}_P\left((H_{\G,{\fc}})|_P ({\bf 1}_{K_P \varpi^{\mu}K_P}) \right)
=0.\qedhere
\end{align*}}

\bibliographystyle{amsalpha}
\bibliography{Reda_library2020}

\end{document}